\documentclass[a4paper,twoside,8pt]{article}

\usepackage[T1]{fontenc}
\usepackage{times}
\usepackage{amsmath,amsfonts,amssymb,amsthm,euscript,pifont}

\usepackage{authblk}
\usepackage{color}
\usepackage[colorlinks=true]{hyperref}
\hypersetup{
    linkcolor = {blue},
    citecolor = {blue}
    }
\usepackage{dsfont}
\usepackage{fullpage}

\usepackage{mathrsfs}
\usepackage{mathtools, stmaryrd}

\usepackage{xparse}
\DeclarePairedDelimiterX{\Iintv}[1]{\llbracket}{\rrbracket}{\iintvargs{#1}}\NewDocumentCommand{\iintvargs}{>{\SplitArgument{1}{,}}m}
{\iintvargsaux#1} %
\NewDocumentCommand{\iintvargsaux}{mm} {#1\mkern1.5mu,\mkern1.5mu#2}
\usepackage{xspace}

\newtheorem{theorem}{Theorem}[section]
\newtheorem{corollary}[theorem]{Corollary}
\newtheorem{definition}[theorem]{Definition}

\newtheorem{lemma}[theorem]{Lemma}
\newtheorem{proposition}[theorem]{Proposition}
\newtheorem{prop}[theorem]{Proposition}
\newtheorem{remark}[theorem]{Remark}

\newcommand{\1}{\mathds{1}}

\newcommand{\er}{\mathbb R}

\newcommand{\Bb}{\mathcal B}
\newcommand{\Cc}{\mathcal C}

\newcommand{\Ff}{\mathcal F}

\newcommand{\Ll}{\mathcal L}

\newcommand{\Pp}{\mathcal P}

\newcommand{\Xx}{\mathcal X}

\newcommand{\EE}{\mathbb E}

\newcommand{\PP}{\mathbb P}

\newcommand{\hypo}{({\bf A0})\xspace}
\newcommand{\hypi}{({\bf A1})\xspace}
\newcommand{\hypii}{({\bf A2})\xspace}
\newcommand{\hypiii}{({\bf A3})\xspace}

\numberwithin{equation}{section}

\title{Rate of propagation of chaos for diffusive stochastic particle systems via Girsanov transformation.\footnote{Preliminary version, July 2019}}
\author{Jean-Fran\c cois JABIR\footnote{School of Mathematics, University of Edinburgh, Scotland; jjabir@ed.ac.uk}\footnote{Higher School of Economics, National Research University, Russia Federation; jjabir@hse.ru}}
\date{\empty}
\begin{document}
 \maketitle

\paragraph{Abstract}: This paper focus on investigating the explicit rate of convergence for the propagation of chaos, in a pathwise sense a family of interacting stochastic particle related to some Brownian driven McKean-Vlasov dynamics. Precisely the McKean form of nonlinearity is concentrated on a path dependent drift component and satisfies a particular sub-gaussian moment control. Such control enables to derive a uniform estimate of the cost in terms of exponential martingale between the particle and its McKean/mean-field limit system which in turn provide an optimal rate of propagation of chaos in terms of the total variation distance. 
  As a by-product, we deepen some recent propagation of chaos results due to Lacker \cite{Lacker-18}
  and provides a partial stochastic interpretation of the entropy control technique introduced in Jabin and Wang \cite{JabinWang-17}.



 \section{Introduction}
Hereafter, we are interested in the explicit rate at which a system of $N$-interacting stochastic particle $(X^{1,N},X^{2,N},\dots,X^{N,N})$ satisfying
 \begin{equation}
 \label{eq:particle_intro}
 \left\{
\begin{aligned}
&X^{i,N}_t=X^i_0+\int_{0}^{t}c(s,(X_r)_{0\leq r\leq s})\,ds\\
&\hspace{1cm}+\int_0^t A(s,(X^{i,N}_r)_{0\leq r\leq s})\Big(B(s,(X^{i,N}_r)_{0\leq r\leq s};\overline{\mu}^{N,N}_s)\,ds+\,dW^i_s\Big),\,i=1,\cdots,N,\,0\leq t\leq T,\\
&\overline{\mu}^{N,N}_t=\frac{1}{N}\sum_{j=1}^N\delta_{\{(X^{j,N}_r)_{0\leq r\leq t}\}},\,X^i_0\sim \mu^0,\,(X^1_0,X^2_0,\dots,X^N_0)\,\text{independent},
\end{aligned}
\right.
\end{equation}
propagates chaos. The particle system is defined up to some finite time horizon $0<T<\infty$, with a given initial distribution on $\er^d$ and $W^1,\dots,W^N$ a sequence of independent $m$-dimensional standard Brownian motions ($m\geq 1$). The system of SDEs \eqref{eq:particle_intro} mainly endows an non-anticipative diffusion component $A$ and two non-anticipative drift components $c$ and $A B$ (resulting from the product of $A$ and $B$) and issued from some given progressively measurable mappings:
\[
c:(t,x)\in[0,T]\times\Cc([0,T];\er^d) \mapsto c(t,x)=c\big(t,(\omega_{\theta\wedge t}(x))_{0\leq \theta\leq T}\big)\in \er^{d},
\]
\[
A:(t,x)\in[0,T]\times\Cc([0,T];\er^d) \mapsto A(t,x)=A\big(t,(\omega_{\theta\wedge t}(x))_{0\leq \theta\leq T}\big)\in\er^{d\times m},
\]
\[
B:(t,x,P)\in[0,T]\times\Cc([0,T];\er^d)\times \Pp(\Cc([0,T];\er^d))\mapsto B(t,x;P)=B\big(t,(\omega_{\theta\wedge t}(x))_{0\leq \theta\leq T},P\circ((\omega_{\theta\wedge t})_{0\leq \theta\leq T})^{-1}\big)\in\er^m,
\]
for $(\omega_t)_{0\leq t\leq T}$ the canonical process on $\Cc([0,T];\er^d)$. In particular, the interaction between particles are described by the component $B$ whose values range in the same dimension as of the Brownian diffusion driving each elements of \eqref{eq:particle_intro}.

The propagation of chaos property will be here mainly understood for the law of the paths of \eqref{eq:particle_intro}; namely in the sense where, for a fixed number of particles $X^{1,N},\dots,X^{k,N}$ as the overall number $N$ of interacting particles increases, the chaos (independency) of the initial $X^1_0,\dots,X^N_0$ and diffusive $(W^1_t)_{t\geq 0},\,\dots,(W^N_t)_{t\geq 0}$ inputs of the system is restored in the particle dynamics of the group of particles yielding to the generic dynamic:
\begin{equation}
 \label{eq:McKeanVlasov_intro}
\left\{
\begin{aligned}
&X^{\infty}_t=X_0+ \int_{0}^{t}c(s,(X^{\infty}_r)_{0\leq r\leq s})\,ds\\
&\hspace{1cm}+\int_0^t A(s,(X^{\infty}_r)_{0\leq r\leq s})\Big(B(s,(X^{\infty}_r)_{0\leq r\leq s};\Ll((X^{\infty}_r)_{0\leq r\leq s}))\,ds+\,dW_s\Big),\,0\leq t\leq T,\\
&\Ll((X^{\infty}_r)_{0\leq r\leq t}))=\text{Law of }((X^\infty_r)_{0\leq r\leq t}),\,X_0\sim \mu^0,
\end{aligned}
\right.
\end{equation}
and the weak limit behaviour:
\[
\Ll((X^{1,N}_t,\dots,X^{k,N}_t)_{0\leq t\leq T})\underset{N\rightarrow \infty}{\longrightarrow} \Ll((X^{\infty}_t)_{0\leq t\leq T})\otimes\dots\otimes \Ll((X^{\infty}_t)_{0\leq t\leq T}).
\]
Due to the exchangeability of the particle system, this property is further equivalent to
\[
\Ll\Big(\frac{1}{N}\sum_{j=1}^N\delta_{\{(X^{k,N}_t)_{0\leq t\leq T} \}}\Big)\underset{N\longrightarrow \infty}{\longrightarrow}\Ll((X_t)_{0\leq t\leq T})\,\text{in the weak sense on }\Pp(\Cc([0,T];\er^d)),
\]
 whenever $k\geq 2$, [Sznitman \cite{Sznitman-89}, Proposition 2.2].
 Particular cases of interest for \eqref{eq:McKeanVlasov_intro} that will be discussed later are the situations where the interaction kernel is of the form
 \[
\int b(t,x,\tilde{x})\,\nu(d\tilde{x}),\,t\geq 0,\,x\in\er^d,\,\nu\in\Pp(\er^d),\,b:[0,\infty)\times\er^d\times\er^d\rightarrow \er^m\,\text{bounded},
 \]
 and where the diffusion component $A$ is either a $d\times d$-valued ($m=d$) bounded and uniformly elliptic matrix or, for $d=2m$, is of the form:
 \begin{equation}\label{eq:LangevinMcKean}
   A=\begin{pmatrix}
     0 & 0\\
     0 & \sigma\\
   \end{pmatrix}
 \end{equation}
More precisely, the former case corresponds to the prototypical McKean-Vlasov dynamic:
\begin{equation}\label{eq:ProtoMcKeanVlasovintro}
dX_t=\Big(\int b(t,(Y_t,V_t),(y,v))\,\mu(t,dx)\Big)\,dt+\sigma(t,X_t)\,dW_t,\,\mu(t)=\Ll(X_t),\,X_0\sim \mu_0,
\end{equation}
while the later case can be further particularized into a Langevin dynamic $X_t=(Y_t,V_t)\in\er^m\times\er^m$ satisfying:
 \begin{equation}\label{eq:LangevinMcKeanintro}
\left\{
\begin{aligned}
&dY_t=V_t\,dt,\,\,(Y_0,V_0)\sim\mu_0\\
&dV_t=\Big(\int b(t,(Y_t,V_t),(y,v))\,\mu(t,dy,dv)\Big)\,dt+\sigma(t,X_t)\,dW_t,\,\mu(t)=\Ll(Y_t,V_t).
\end{aligned}
\right.
\end{equation}

 The propagation of chaos property of stochastic interacting particle systems has received over of the years a tremendous amount of attention since its initial introduction in statistical physics (Kac \cite{Kac-56}) for its applications for the probabilistic interpretation of nonlinear pdes (McKean \cite{McKean-66}, \cite{McKean-67}; see the surveys Bossy \cite{Bossy-03}, Jabin and Wang \cite{JabinWang-17} for two global overviews on the theoretical and practical aspects related to McKean-Vlasov or McKean SDEs and related particles approximations) and in its modern utilization for the description of interacting economical agents models and game theory (see e.g. Kolokolstov \cite{Kolokolstov-10}, Carmona and Delarue \cite{CarDel-18a}, \cite{CarDel-18b} and references therein).

 The central result of the present paper (Theorem \ref{mainthm1:LinearCaseTV}) establishes an explicit (and optimal) rate of convergence for the propagation of chaos property between \eqref{eq:particle_intro} and \eqref{eq:McKeanVlasov_intro} in terms of the total variation distance:
 \[
 \Vert\mu-\nu\Vert_{TV}=\sup_{A\in\Bb(\Cc([0,T];\er^d))}\left|\int \1_{\{x\in A\}}\mu(dx)-\int \1_{\{x\in A\}}\nu(dx)\right|,\,\mu,\nu\in\Pp(\Cc([0,T];\er^d)),
 \]
 Mainly this result rests on a generic criterion (see the condition $(\mathbf{C})$ below) which does not directly relies on some regularity properties of $B$ but rather ensure the control of some moments of the Doleans-Dale exponential martingale related to the Girsanov transformation which maps the $N$-system of McKean SDEs \eqref{eq:McKeanVlasovParticle} into the $N$-interacting particle system \eqref{eq:Nparticles}.
  
 The core idea of the main result of the present paper is based on a probabilistic interpretation of the proof techniques introduced in Jabin and Wang \cite{JabinWang-16} for the propagation of chaos in entropy (and by extension in total variation) of the one time-marginal distributions of McKean-Vlasov dynamics of the form \eqref{eq:LangevinMcKeanintro} with bounded interaction. More generally, the authors designed a guideline for establishing a sharp quantitative estimate of the propagation of chaos, in terms of a vanishing initial chaos (the particle being initially correlated) and (possibly) vanishing diffusion, through a powerful combination of pde analysis, entropy estimate and combinatorics. This guideline, combined with large deviations principles, was extended to the instance of McKean-Vlasov dynamics \eqref{eq:McKeanVlasov_intro} endowed with singular interaction kernels of the form $b\in W^{-1,\infty}$ (i.e. $b^{(k)}(x)=\sum_{l}\partial_{x_l} G^{k,l}(x),\,G\in L^\infty$).
 
  Linked to the probabilistic interpretation of the proof techniques of \cite{JabinWang-16}, let us mention that a (non-explicit) propagation of chaos property in entropy and in total variation distance was recently considered in Lacker \cite{Lacker-18} for the McKean SDE:
 \begin{equation}
\label{eq:McKeanVlasovPastDepend}
z_t=Z_0+\int_0^t B\big(s,(z_r)_{0\leq r\leq s},\Ll((z_r)_{0\leq r\leq s})\big)\,ds+\int_{0}^{t}\sigma(s,(z_r)_{0\leq r\leq s})\,dW_s,\,0\leq t\leq T.
\end{equation}
and its related particle approximation:
 \begin{equation}
\label{eq:ParticlePastDepend}
z^{i,N}_t=Z^i_0+\int_0^t B\big(s,(z^{i,N}_r)_{0\leq r\leq s},\frac{1}{N}\sum_{j=1}^N\delta_{\{(z^{j,N}_r)_{0\leq r\leq s}\}}\big)\,ds+\int_{0}^{t}\sigma(s,(z^{i,N}_r)_{0\leq r\leq s})\,dW^i_s,\,0\leq t\leq T,
\end{equation}
assuming the uniform ellipticity of $\sigma$, the boundedness and Lipschitz continuity (in terms of the total variation distance) of $\sigma^{-1}B$ and the continuity of
\[
\nu\in\Pp(\Cc([0,T];\er^d))\mapsto \int_{\Cc([0,T];\er^d)}\int_{0}^{T} \left|\sigma^{-1}(t,z)\left( B(t,z,\mu)- B(t,z,\nu)\right)\right|^2\,dt\,\nu(dz),
\]

The core idea of \cite{Lacker-18} is closely connected to the original idea introduced in Mishura and Veretennikov \cite{MisVer-16} (from which the present paper owns also its initial step) linking the measurement of the total variation distance between two It\^o's diffusion processes
in terms of the Girsanov transformation between the two processes and its applications for the weak uniqueness problems of the McKean SDEs \eqref{eq:ProtoMcKeanVlasovintro}. (It should also be noticed that the idea of establishing propagation of chaos through the Girsanov transformation was already hinted in the preprint Veretennikov \cite{Veretennikov-18} almost at the same time as \cite{Lacker-18}.) The dynamics \eqref{eq:particle_intro} and \eqref{eq:McKeanVlasov_intro} considered hereafter present a extended version of \eqref{eq:McKeanVlasovPastDepend} and \eqref{eq:ParticlePastDepend} which enable to relax elliptic assumption on the diffusion coefficients and embed the case \eqref{eq:LangevinMcKeanintro}. Let also mention that, compared to \cite{Lacker-18}, the wellposed problems related to \eqref{eq:McKeanVlasov_intro} and \eqref{eq:particle_intro} will not be addressed hereafter (assumptions \hypi and \hypii) to rather focus on quantifying explicitly the related propagation of chaos property.

 The main result of this paper (Theorem \ref{mainthm1:LinearCaseTV}) is stated in Section \ref{sec:MainResults} and proved in Section \ref{sec:Proof}. Section \ref{sec:SufficientConditions} is dedicated to applications of this main result in the particular cases \eqref{eq:ProtoMcKeanVlasovintro} and  \eqref{eq:LangevinMcKeanintro} (see corollaries \ref{coro:BoundedCase} and \ref{coro:KineticBoundedCase} respectively) and to exhibit a sufficient condition for the condition $(\mathbf{C})$ in terms of the second order differentiability of $\nu\mapsto B(t,x,\nu)$ (Proposition \ref{prop:DifferentiabilityCondition}). Although \eqref{eq:ProtoMcKeanVlasovintro} and \eqref{eq:LangevinMcKeanintro} only presents applications of Theorem \ref{mainthm1:LinearCaseTV} where the interaction is bounded, more singular situations should be handled by cut-smoothing techniques. The particular case of conditional McKean Lagrangian models (see Bossy, Jabir and Talay \cite{jabir-11a}), which initially motivated the present work, will be discussed in \cite{JabMen-19}.

\textbf{Assumptions}:
(As before, $(A B)$ denotes the functional on $[0,\infty)\times\Cc([0,\infty);\er^d)\times \Pp(\Cc([0,\infty);\er^d))$ resulting from the product between the diffusion $A$ and drift component $B$ in \eqref{eq:McKeanVlasovParticle} and \eqref{eq:McKeanVlasov_intro}.)

\noindent
\hypo For any $\mu_0$ on $\er^d$, $0\leq T<\infty$, there exists a unique weak solution $(\Xx_t)_{0\leq t\leq T}$ satisfying the SDE:
 \begin{equation}\label{eq:IntermediateSDE}
\left\{
\begin{aligned}
&\Xx_t=X_0+ \int_{0}^{t}c(s,(\Xx_r)_{0\leq r\leq s})\,ds+\int_0^t A(s,(\Xx_r)_{0\leq r\leq s})\,dW_s,\,0\leq t\leq T,\\
&\Xx_0\sim \mu^0.
\end{aligned}
\right.
\end{equation}
and for $(\Xx^{1}_t)_{0\leq t\leq T},\,\dots,(\Xx^{1}_t)_{0\leq t\leq T}$ a family of $N$ independent copies of $(\Xx_t)_{0\leq t\leq T}$, it holds that
$1\leq i\leq N$,
\[
\int_0^T\left|  \big (A B\big)\big(s,(\Xx^{i}_r)_{0\leq r\leq s};\nu^N_s\big)\right|^2\,ds<\infty,
\]
where $\nu^{N}_t=\frac{1}{N}\sum_{j=1}^{N}\delta_{\{(\Xx^j_r)_{0\leq r\leq t}\}}$.

\noindent
\hypi For any $\mu_0$, $0<T<\infty$, the SDE \eqref{eq:McKeanVlasov_intro} admits a unique weak solution $(X_t)_{t\geq 0}$ such that, almost surely,
\[
\int_0^T\left|\big(A B\big)(s,(X_r)_{0\leq r\leq s};\Ll((X_r)_{0\leq r\leq s}))\right|^2\,ds<\infty.
\]

\noindent
\hypii For any $\mu_0$, $0<T<\infty$, $N\geq 1$, the system of SDEs \eqref{eq:particle_intro} admits a unique weak solution
$\{(X^{i,N}_t)_{t\geq 0};\,1\leq i\leq N\}$ such that, a.s.
\[
\forall\,1\leq i\leq N,\,\int_0^T\left|\big(A B\big)(s,(X^{i,N}_r)_{0\leq r\leq s};\overline{\mu}^{N,N}_s)\right|^2\,ds<\infty,
\]
where $(\overline{\mu}^{N,N}_t)_{0\leq t\leq T}$ is the flow of (random) empirical measures given as in \eqref{eq:McKeanVlasov_intro}.

\begin{remark}
With the assumptions \hypi and \hypii, we deliberately leave aside the wellposedness problems of a weak solution to the $N$-interacting particle system \eqref{eq:particle_intro} and to the McKean SDE \eqref{eq:McKeanVlasov_intro} to rather focus on quantifying explicitly the related propagation of chaos property. Although not necessary, the assumption \hypo is used to ensure, in a simple way, the equivalency in law between \eqref{eq:particle_intro} and \eqref{eq:McKeanVlasov_intro}. Let us also mention that the assumptions on the weak uniqueness of \eqref{eq:particle_intro} and \eqref{eq:McKeanVlasov_intro} can be relaxed as long as there exist a solution to \eqref{eq:particle_intro} and a solution to \eqref{eq:McKeanVlasov_intro} for which \eqref{proofstp:i} hold.
\end{remark}

\textbf{Notation:}
For any integer $m\geq 1$, and any finite positive time horizon $T$, $\Cc([0,T];\er^{m})$ (respectively $\Cc([0,\infty);\er^{m})$) will denote the space of continuous functions defined on $[0,T]$ (resp. $[0,\infty)$) with values in $\er^m$ equipped with the uniform norm $\Vert x\Vert_{\Cc([0,T];\er^m)}=\max_{0\leq t\leq T}|x(t)|$ (resp. $\Vert x\Vert_{\Cc([0,\infty);\er^m)}=\max_{t\geq 0}|x(t)|\wedge 1)$. $\Pp(\Cc([0,T];\er^{m}))$ and $\Pp(\Cc([0,\infty);\er^{m}))$ will denote respectively the space of probability measures defined on $\Cc([0,T];\er^m)$ and on $\Cc([0,\infty);\er^m)$. Finally,
$\Vert~\Vert_{TV,(0,T)}$ will denote the total variation norm on $\Pp(\Cc([0,T];\er^{m}))$, that is (see e.g. Equation $(3.2.13)$ in Rachev \cite{Rachev-91}): for all $P_1,P_2$ on $\Pp(\Cc([0,T];\er^{m}))$
\[
\Vert P_1-P_2\Vert_{TV,(0,T)}=\sup_{A \in \Bb(\Cc([0,T];\er^m))}\left|P_1(A)-P_2(A)\right|,
\]
where $\Bb(\Cc([0,T];\er^m))$ denotes the Borel $\sigma$-algebra of $\Cc([0,T];\er^m)$. Whenever $P_1, P_2\in
\Pp(\Cc([0,\infty);\er^{m}))$ and $0<T<\infty$ is a finite time horizon, $\Vert P_1-P_2\Vert_{TV,(0,T)}$ will simply correspond to the total variation distance between the probability measures restrained to the sample space $(\Cc([0,T];\er^d),\Bb(\Cc([0,T];\er^d)))$.

\section{Main result}\label{sec:MainResults}
Let $(\Omega,\Ff,(\Ff_t;\,0\leq t\leq T),\PP)$ and $(\widetilde{\Omega},\widetilde{\Ff},(\widetilde{\Ff}_t;\,0\leq t\leq  T),\widetilde{\PP})$ be two (possibly different) filtered probability spaces under each of which are defined a collection of $(X^i_0,(W^i_t)_{0\leq t\leq T})$ and $(\widetilde{X}^i_0,(\widetilde{W}^i_t)_{0\leq t\leq T})$ of independent copies of $(X_0,(W_t)_{0\leq t\leq T})$. Then, under  \hypi and \hypii, consider a version of the particle system \eqref{eq:McKeanVlasovParticle} defined on $(\widetilde{\Omega},\widetilde{\Ff},(\widetilde{\Ff}_t;\,0\leq t\leq  T),\widetilde{\PP})$ as
\begin{equation}
 \label{eq:Nparticles}
 \left\{
\begin{aligned}
&X^{i,N}_t=\widetilde{X}^i_0+\int_{0}^{t}c(s,(X_r)_{r\leq s})\,ds\\
&\hspace{1cm}+\int_0^t A(s,(X^{i,N}_r)_{0\leq r\leq s})\Big(B(s,(X^{i,N}_r)_{0\leq r\leq s};\overline{\mu}^{N,N}_s)\,ds+d\widetilde{W}^i_s\Big),\,0\leq t\leq T,\,i=1,\cdots,N,\\
&\overline{\mu}^{N,N}_t=\frac{1}{N}\sum_{j=1}^N\delta_{\{(X^{j,N}_r)_{0\leq r\leq t}\}},\,\widetilde{X}^i_0\sim \mu^0,
\end{aligned}
\right.
\end{equation}
and a system of $N$-independent copies of \eqref{eq:McKeanVlasov_intro} defined on $(\Omega,\Ff,(\Ff_t;\,0\leq t\leq T),\PP)$ as
\begin{equation}
 \label{eq:McKeanVlasovParticle}
 \left\{
\begin{aligned}
&X^{i,\infty}_t=X^i_0+\int_{0}^{t}c(s,(X^{i,\infty}_r)_{r\leq s})\,ds\\
&\hspace{1cm}+\int_0^t A(s,(X^{i,\infty}_r)_{0\leq r\leq s})\Big(B(s,(X^{i,\infty}_r)_{0\leq r\leq s};\Ll((X^{i,\infty}_r)_{0\leq r\leq s}))\,ds+\,dW^i_s\Big),\\
&\mu^{i,\infty}(t)=\Ll((X^{i,\infty}_r)_{0\leq r\leq t}),\,X^i_0\sim \mu^0.
\end{aligned}
\right.
\end{equation}
As the assumption \hypii ensures the uniqueness of each component of the system \eqref{eq:McKeanVlasovParticle}, the distribution $\Ll((X^{i,\infty}_t)_{0\leq t\leq T})$ is the common for all component and equal to the one of \eqref{eq:McKeanVlasov_intro}; the index $i$ may be dropped. The superscript $\infty$ in \eqref{eq:McKeanVlasovParticle} will be used as a pointer to remind that \eqref{eq:McKeanVlasovParticle} is (at least heuristically) the suitable limit system of \eqref{eq:Nparticles}.

Our main result is given by the following theorem:

\begin{theorem}\label{mainthm1:LinearCaseTV} Assume that \hypi and \hypii hold. Assume also that the following condition $\mathbf{(C)}$ holds:
\begin{equation*}
\mathbf{(C)}\,\,\left\|\,\,
\begin{aligned}
&\text{There exists a constant }0<\beta<\infty\text{ such that for any }0<T_0<T<\infty,\,0<\delta<\infty,\text{and, for all integer }p\geq 1,\\
&\hspace{4cm}\EE_{\PP}\left[\left(\int_{T_0}^{(T_0+\delta)\wedge T}\left|  \triangle B^{i,N,\infty}_t\right|^{2}\,dt\right)^p\right]\leq \frac{p! \beta^{p}\delta^p}{N^p},\\
&\text{where}\,\triangle B^{N,\infty}_t=B(t,(X^{i,\infty}_r)_{0\leq r\leq t};\overline{\mu}^{N,\infty}_t)- B(t,(X^{i,\infty}_r)_{0\leq r\leq t};\Ll((X^{i,\infty}_r)_{0\leq r\leq t}),\\
&\overline{\mu}^{N,\infty}_t=\frac{1}{N}\sum_{j=1}^N\delta_{\{(X^{j,\infty}_r)_{0\leq r\leq t}\}}.
\end{aligned}
\right.
\end{equation*}
Then
\[
\Vert \Ll\big((X^{1,N}_t,X^{2,N}_t,\dots,X^{N,N}_t)_{0\leq t\leq T}\big)- \Ll\big((X^{1,\infty}_t,X^{2,\infty}_t,\dots,X^{N,\infty}_t)_{0\leq t\leq T}\big)\Vert_{TV,(0,T)}\leq C(1+\beta T)\sqrt{\frac{k}{N}},
\]
where $C$ is a constant only depending on $T$, $m$ and $\beta$.
\end{theorem}
The condition $\mathbf{(C)}$ can be understood as a local Novikov condition in the spirit the one key argument for the proof of Khasm'inskii's lemma (see e.g. [Simon \cite{Simon-82}, Lemma B.1.2.]). Alternatively the condition $\mathbf{(C)}$ in Theorem \ref{mainthm1:LinearCaseTV} can be viewed as a (non-asymptotic) large deviation principle or a sub-gaussian concentration property for the deviation between the "empirical" drift of \eqref{eq:Nparticles} evaluated along the $N$-system of McKean SDEs \eqref{eq:McKeanVlasovParticle}:
\[
 B\big(t,((X^{i,\infty}_r)_{0\leq r\leq t});\frac{1}{N}\sum_{j=1}^N\delta_{\{(X^{j,\infty_r})_{0\leq r\leq t}\}}\big),
 \]
 and its mean-field limit:
 \[
 B\big(t,((X^{i,\infty}_r)_{0\leq r\leq t});\Ll((X^{i,\infty_r})_{0\leq r\leq t})\big).
 \]
 In the situations \eqref{eq:ProtoMcKeanVlasovintro} and \eqref{eq:LangevinMcKeanintro}, $\mathbf{(C)}$ is a direct consequence of the boundedness of the interaction kernel $b$. In more general situation the condition may result from a Lipschitz property of $\nu\in\Pp(\Cc([0,T];\er^d))\mapsto B\big(t,x;\nu\big)$ and a centering property (see Lemma \ref{lem:RatePathDependent}) or from a higher regularity property in terms of the variational- linear functional derivative of  $\nu\in\Pp(\Cc([0,T];\er^d))\mapsto B\big(t,x;\nu\big)$ (see Definition \ref{def:FlatDerivative} and Proposition \ref{prop:DifferentiabilityCondition}).

\section{Proof of Theorem \ref{mainthm1:LinearCaseTV}}\label{sec:Proof}
\subsection{Preliminary on propagation of chaos for the total variation distance and control of the Girsanov transformation between $\Ll(X^{1,\infty},\dots,X^{N,\infty})$ and $\Ll(X^{1,N},\dots,X^{N,N})$}
For notation convenience, define
\[
P^{k,N}=\Ll((X^{1,N}_t,X^{2,N}_t,\cdots,X^{k,N}_t)_{0\leq t\leq T})\in \Pp(\Cc([0,T];\er^{dk})),
\]
 the joint law of the first $k$ particles of \eqref{eq:McKeanVlasovParticle} and by
\[
P^{k,\infty}=\Ll((X^{1,\infty}_t,X^{2,\infty}_t,\cdots,X^{k,\infty}_t)_{0\leq t\leq T})\in \Pp(\Cc([0,T];\er^{dk})),
\]
the joint law of the first $k$ independents copies of \eqref{eq:McKeanVlasov_intro}. The later reduces to
\[
P^{k,\infty}=\underbrace{P^{\infty}\otimes P^{\infty}\otimes  \cdots \otimes P^{\infty}}_{\text{k times}},\,\,\, P^{\infty}=\Ll\big((X^{\infty}_t)_{0\leq t\leq T}\big),
\]
as the assumption \hypi ensures the weak uniqueness of  \eqref{eq:McKeanVlasov_intro}.

The combination of the assumptions \hypo, \hypi and \hypii ensure that for all $1\leq k\leq N<\infty$, the measures $P^{k,N}$ and $P^{k,\infty}$ are equivalent and the Radon-Nikodym derivative formulates\footnote{The proof of \eqref{proofstp:i} under the sole assumptions \hypo, \hypi and \hypii is detailed in the appendix section.} is given by the Doleans-Dale exponential martingale:
 \begin{equation}
\label{proofstp:i}
\begin{aligned}
&Z^{N}_T:=\frac{dP^{N,N}}{dP^{N,\infty}}\\
&=\exp\left\{-\sum_{i=1}^N\int_0^T \left(B\big(t,(X^{i,\infty}_r)_{0\leq r\leq t},\frac{1}{N}\sum_{j=1}^N\delta_{\{(X^{j,\infty}_r)_{0\leq r\leq t}\} }\big)-\int B\big(t,(X^{i,\infty}_r)_{0\leq r\leq t},\Ll((X^{i,\infty}_r)_{0\leq r\leq t})\big)\right)\cdot \,dW^{i}_t\right.\\
&\quad \left.-\frac{1}{2}\int_0^T  \sum_{i=1}^N
\left|B\big(t,(X^{i,\infty}_r)_{0\leq r\leq t},\frac{1}{N}\sum_{j=1}^N\delta_{\{(X^{j,\infty}_r)_{0\leq r\leq t}\}}\big)-\int B\big(t,(X^{i,\infty}_r)_{0\leq r\leq t},\Ll((X^{i,\infty}_r)_{0\leq r\leq t})\big)
\right|^2\,dt\right\}\\
&=\exp\left\{-\sum_{i=1}^N\int_0^T \triangle B^{i,N}_t\cdot \,dW^{i}_t-\frac{1}{2}\sum_{i=1}^N\int_0^T \left|\triangle B^{i,N}_t\right|^2\,dt\right\},
\end{aligned}
\end{equation}
where $(\triangle B^{i,N}_t)_{0\leq t\leq T},\,i=1\dots,N$ are given as in $(\mathbf{C})$. By Csisz\'ar-Pinsker-Kullback's inequality,
\begin{equation}\label{proofstp:g}
\Vert P^{k,N}-P^{k,\infty}\Vert_{TV,\Pp((\Cc([0,T];\er^{kd})))} \leq \sqrt{2 H(P^{k,N}\,|\,P^{k,\infty})},
\end{equation}
where $H(P^{k,N}\,|\,P^{k,\infty})$ is the relative entropy between $P^{k,N}$ and $P^{k,\infty}$ is given by
 \[
 H(P^{k,N}\,|\,P^{k,\infty})=\int_{\mathbf{\omega}^k\in\Cc([0,T];\er^{dk})} \log(dP^{k,\infty}/dP^{k,\infty})(\mathbf{\omega}^{k})P^{k,N}(d\mathbf{\omega}^k)
 \]
 with $dP^{k,N}/dP^{k,\infty}$ being explicitly given by the conditional expectation $\EE_{\PP}\left[Z^N_T \,|\,(X^{1,\infty},\dots,X^{k,\infty})\right]$ valuing the average value of $Z^{N}_T$ given the path on $[0,T]$ of the $k$-first components of \eqref{eq:McKeanVlasovParticle}, $(X^{1,\infty}_t,\dots,X^{k,\infty}_t)_{0\leq t\leq T}$. At this stage, for $1\leq k<N$, decomposing the empirical measure $\frac{1}{N}\sum_{j=1}^N\delta_{\{(X^{j,\infty}_r)_{0\leq r\leq t}\}}$ into
  \[
  \frac{1}{N}\sum_{j=1}^k\delta_{\{(X^{j,\infty}_r)_{0\leq r\leq t}\}}+\frac{N-(k+1)}{N}\left(\frac{1}{N-(k+1)}\sum_{j=k+1}^N\delta_{\{(X^{j,\infty}_r)_{0\leq r\leq t}\}}\right),
  \]
  and owing to the l.s.c. property of $H$ and as $(X^{k+1,\infty},\dots,X^{N,\infty})$ are i.i.d.,
  a natural propagation of chaos property can be derived providing some boundedness and continuity properties on $\nu\mapsto B(t,x;\nu)$. (In \cite{Lacker-18}, an alternative route was proposed  proving that $\lim_{N\rightarrow\infty}H(P^{k,\infty}\,|\,P^{k,N})=0$. This results was derived succeeding from a preliminary propagation of chaos results $\frac{1}{N}\sum_{j=1}^N\delta_{\{(X^{j,N}_r)_{0\leq r\leq T}\}}\rightarrow \Ll((X^\infty_r)_{0\leq r\leq T})$ derived from a large deviation principle.) An explicit estimate of the propagation of chaos can further be deduced from the super-additive property of the renormalized relative entropy (see e.g. [Hauray and Mischler 2014, Lemma 3.3-iv]),
\[
\frac{1}{k}H\big(P^{k,N}|\,P^{k,\infty}\big)\leq \frac{1}{N}H\big(P^{N,N}\,|\,P^{N,\infty}\big).
\]
Plugged into \eqref{proofstp:g},
\begin{equation*}
\Vert P^{k,N}-P^{k,\infty}\Vert_{TV,\Pp((\Cc([0,T];\er^{kd})))} \leq \sqrt{\frac{2k}{N} H(P^{N,N}\,|\,P^{N,\infty})}=\sqrt{\frac{2k}{N} \EE_{\PP}\left[Z^{N}_T\log(Z^{N}_T)\right]}.
\end{equation*}
from which emerges the optimal rate $1/\sqrt{N}$ provided $\sup_N\EE[(Z^N_T)^{1+\delta}]<\infty$, for some $\delta>0$.
The necessity of the uniform control for a moment greater than $1$ of $(Z^N_t)_{0\leq t\leq T}$ can be observed more directly in the case of $P^{k,N}$ and $P^{k,\infty}$:
Under \hypi and \hypii, the total variation distance between $P^{k,N}$ and $P^{k,\infty}$ can be expressed as:\eqref{proofstp:i}, for all $A\in\Bb(\Cc([0,T];\er^{kd}))$, we have
\begin{align*}
\widetilde{\PP}((X^{1,N},\dots,X^{k,N})\in A)=P^{k,N}(A)=\EE_\PP\left[Z^{N}_T\1_{\{(X^{1,\infty},\dots,X^{k,\infty})\in A\}}\right]
\end{align*}
from which we deduce that
\begin{align*}
\Vert P^{k,N}-P^{k,\infty} \Vert_{TV,(0,T)}
&=\sup_{A \in\Bb(\Cc([0,T];\er^{2d}))}
\left|\EE_\PP\left[\left(Z^{N}_T-1\right)\1_{\{(X^{1,\infty},\dots,X^{k,\infty})\in A\}}\right]\right|\\
&=\EE_\PP\left[\left|\EE_\PP\left[\left(Z^{N}_T-1\right)\1_{\{(X^{1,\infty},\dots,X^{k,\infty})\in A\}}\,|\,(X^{1,\infty},\dots,X^{k,\infty})\right]\right|\right].
\end{align*}
Since
\[
Z^N_T=1+\sum_{i=1}^{N}\int_0^T Z^N_t\triangle B^{i,N}_t\,\cdot dW^{i}_t=\sum_{i=1}^{N} \sum_{l=1}^m \int_0^T Z^N_t\triangle B^{i,N,(l)}_t\,\cdot dW^{i,(l)}_t,
\]
for
\[
Z^N_t=\frac{dP^{N,N}}{dP^{N,\infty}}\Big{|}_{\Bb(\Cc([0,T];\er^d))}=\exp\left\{-\sum_{i=1}^N\int_0^t \triangle B^{i,N}_r\cdot \,dW^{i}_r-\frac{1}{2}\sum_{i=1}^N\int_0^t \left|\triangle B^{i,N}_r\right|^2\,dr\right\},
\]
and since $(W^{k+1},\dots,W^{N})$ are independent from $(X^{1,\infty},\dots,X^{N,\infty})$, the conditional expectation
\[
\EE_\PP\left[\left(Z^{N}_T-1\right)\1_{\{(X^{1,\infty},\dots,X^{k,\infty})\in A\}}\,|\,(X^{1,\infty},\dots,X^{k,\infty})\right],
\]
reduces into
\[
\EE_\PP\left[\left(\sum_{i=1}^k\int_0^T Z^N_t\triangle B^{i,N}_t\,\cdot dW^{i}_t\right)\,|\,(X^{1,\infty},\dots,X^{k,\infty})\right].
\]
This gives:
\begin{equation}\label{eq:BoundTVa}
\Vert P^{k,N}-P^{k,\infty} \Vert_{TV,(0,T)}= \EE_{\PP}\left[\left|\EE_{\PP}\left[\sum_{i=1}^k\int_0^T Z^{N}_t \triangle B^{i,N}_t\cdot \,dW^{i}_t\,\Big{|}\,(X^{1,\infty}_r,\dots,X^{k,\infty}_r)_{0\leq r\leq T}\right]\right|\right].
\end{equation}
Using successively Burkh\"older-Davis-Gundy's inequality, Jensen's inequality, the exchangeability of $(X^{1,\infty},\dots,X^{N,\infty})$ and H\"older's inequality for an arbitrary $1<p< \infty$, it follows:
\begin{align*}
\Vert P^{k,N}-P^{k,\infty} \Vert_{TV,(0,T)}&\leq \EE_{\PP}\left[\left(\int_0^T (Z^{N}_t)^2 \sum_{i=1}^k\left|\triangle B^{i,N}_t\right|^2\,dt\right)^{1/2}\right]\leq \sqrt{k}\EE_{\PP}\left[\left(\int_0^T (Z^{N}_t)^2\left|\triangle B^{i,N}_t\right|^2\,dt\right)^{1/2}\right]\\
&\leq \sqrt{k}\EE_{\PP}\left[\max_{0\leq t\leq T}(Z^{N}_t)\left(\int_0^T\left|\triangle B^{i,N}_t\right|^2\,dt\right)^{1/2}\right]\\
&\leq \sqrt{k}\left(\EE_{\PP}\left[\max_{0\leq t\leq T}(Z^{N}_t)^p\right]\right)^{1/p}\left(\EE_{\PP}\left[\left(\int_0^T\left|\triangle B^{i,N}_t\right|^2\,dt\right)^{p/(2(p-1))}\right]\right)^{(p-1)/p}.
\end{align*}
Applying Doob's inequality, we get
\begin{equation}\label{eq:BoundTVb}
\Vert P^{k,N}-P^{k,\infty} \Vert_{TV,(0,T)}\leq \sqrt{k}\frac{p}{p-1}\left(\EE_{\PP}\left[(Z^{N}_T)^p\right]\right)^{1/p}\left(\EE_{\PP}\left[\left(\int_0^T\left|\triangle B^{i,N}_t\right|^2\,dt\right)^{p/(2(p-1))}\right]\right)^{(p-1)/p}.
\end{equation}
The display of the rate $1/\sqrt{N}$ is then directly related to the technical difficulty of controlling uniformly a $1+\delta$-moment of $Z^N_T$ as such uniform control would imply that the finiteness of the moments , $\EE_{\PP}[(\sum_{i=1}^N\int_0^T|\triangle B^{i,N}_t|^2\,dt)^{k}]$, which, owing to the exchangeability of $(X^{1,\infty},\dots,X^{N,\infty})$ amounts to establishing $\EE_{\PP}[(\int_0^T|\triangle B^{i,N}_t|^2\,dt)^{k}]$ is of order $1/N^k$.

The proof of Theorem of \ref{mainthm1:LinearCaseTV} below is set by first establishing a local-in-time control of an arbitrary moment of $(Z^N_t)_{0\leq t\leq T}$, which combined with \eqref{eq:BoundTVb} and a careful split of the transformation from $(X^{1,\infty},\dots,X^{N,\infty})$ to $(X^{1,N},\dots,X^{N,N})$ to small time intervals enable to conclude the claim.

\subsection{Proof of Theorem \ref{mainthm1:LinearCaseTV}}

%

\begin{prop}\label{prop:ControlExpMart} Let $\{(X^{i,\infty}_t)_{0\leq t\leq T};\,1\leq i\leq N\}$ be given as in \eqref{eq:McKeanVlasovParticle} and assume that $\mathbf{(C)}$ hold true. Then, for all $0<T_0<T<\infty$, $0<\kappa<\infty$,
  \[
  \sup_N \EE_\PP\left[(Z^N_{T_0+\delta}/Z^N_{T_0})^\kappa\right]=\sup_N\EE_\PP\left[\exp\left\{\kappa\sum_{i=1}^N \int_{T_0}^{T_0+\delta} \triangle B^{i,N}_t\cdot \,dW^{i}_t-\frac{\kappa}{2}\int_{T_0}^{T_0+\delta} \left|\triangle B^{i,N}_t\right|^2 \,dt\right\}\right],
  \]
  is bounded from above by $1+\exp{\kappa^2}+\frac{2}{1-8 \kappa\delta \beta}$ provided that $\delta< (8\kappa \beta)^{-1}$.
\end{prop}

\begin{proof}[Proof of Proposition \ref{prop:ControlExpMart}] For the moment, let $\delta$ be an arbitrary positive real number and let us show that
 \[
  \sup_N\EE_\PP\left[\exp\left\{\kappa\sum_{i=1}^N \int_{T_0}^{T_0+\delta} \triangle B^{i,N}_t\cdot \,dW^{i}_t\right\}\right]<\infty.
  \]
 Using the Taylor expansion for the exponential function,
\begin{align*}
\EE_\PP\left[ \exp\left\{\kappa\sum_{i=1}^N\int_{T_0}^{T_0+\delta}\triangle B^{i,N}_t\cdot \,dW^{i}_t\right\}\right]
\leq \sum_{k\geq 0}\frac{\kappa^k}{k!}\EE_\PP\left[\left(\sum_{i=1}^N\int_{T_0}^{T_0+\delta} \triangle B^{i,N}_t\cdot \,dW^{i}_t\right)^k\right].
\end{align*}
Splitting this sum into its even and odd components, and since, for all $r\in\er$, $r^{2p+1}\leq 1+r^{2p+2}$, we have
\begin{equation}
\label{proofstp:a}
\begin{aligned}
&\EE_\PP\left[ \exp\left\{\kappa\sum_{i=1}^N\int_{T_0}^{T_0+\delta}\triangle B^{i,N}_t\cdot \,dW^{i}_t\right\}\right]\\
&\leq \sum_{p\geq 0}\frac{\kappa^{2p+1}}{(2p+1)!}\EE_\PP\left[\left(\sum_{i=1}^N\int_{T_0}^{T_0+\delta}  \triangle B^{i,N}_t\cdot \,dW^{i}_t\right)^{2p+1}\right] +\sum_{p\geq 0}\frac{\kappa^{2p}}{(2p)!}\EE_\PP\left[\left(\sum_{i=1}^N\int_{T_0}^{T_0+\delta}  \triangle B^{i,N}_t\cdot \,dW^{i}_t\right)^{2p}\right]\\
&\leq 1+\sum_{p\geq 0}\frac{\kappa^{2p+1}}{(2p+1)!}+2\sum_{p\geq 0}\frac{|\kappa|^{2p}}{(2p)!}\EE_\PP\left[\left(\sum_{i=1}^N\int_{T_0}^{T_0+\delta} \triangle B^{i,N}_t\cdot \,dW^{i}_t\right)^{2p}\right].
\end{aligned}
\end{equation}
Applying the martingale moment control of Carlen-Kr\'ee \cite{CarKre-91} (see Theorem \ref{thm:CarlenKree}, Appendix section, for a reminder), we have
\begin{align*}
 \EE_\PP\left[\left(\sum_{i=1}^N\int_{T_0}^{T_0+\delta} \triangle B^{i,N}_t\cdot \,dW^{i}_t\right)^{2p}\right]
 \leq 2^{2p} (2p)^{p} \EE_\PP\left[\left(\sum_{i=1}^N\int_{T_0}^{T_0+\delta}  \left|\triangle B^{i,N}_t\right|^2\,dt\right)^{p}\right].
 \end{align*}
Then, by Jensen's inequality and the exchangeability of the $N$-system of McKean-Vlasov dynamics, we get that
\begin{align*}
\EE_\PP\left[\left(\sum_{i=1}^N\int_{T_0}^{T_0+\delta}  \left|\triangle B^{i,N}_t\right|^2\,dt\right)^{p}\right]
& \leq N^{p}\EE_\PP\left[\left(
\int_{T_0}^{T_0+\delta}  \left|\triangle B^{i,N}_t\right|^2\,dt\right)^{p}\right].
%
%
\end{align*}
Plugin the estimate of the condition $\mathbf{(C)}$ then ensures the upper bound
\begin{equation}
\label{proofstp:c}
\begin{aligned}
\EE_\PP\left[ \exp\left\{\kappa\sum_{i=1}^N\int_{T_0}^{T_0+\delta}\triangle B^{i,N}_t\cdot \,dW^{i}_t\right\}\right]\leq 1+\exp{\kappa^2}+2\sum_{p\geq 0}\frac{p!p^p2^{3p}\delta^p\beta^{p}\kappa^p}{(2p)!}.
\end{aligned}
\end{equation}
Since $C:=\sup_{p}\big(p!p^p/((2p)!) \big)<\infty$, the sum is essentially geometric and the condition $\delta/(8\beta\kappa)<1$ ensures its finiteness with
\[
\sup_N \EE_\PP\left[(Z^N_{T_0+\delta}/Z^N_{T_0})^\kappa\right]\leq 1+\exp{\kappa^2}+\frac{2}{1-8 \kappa\delta \beta} .
\]
\end{proof}

Coming back to the proof of Theorem \ref{mainthm1:LinearCaseTV}], for an arbitrary integer $1<p<\infty$, and for $\overline{\delta}:=(8\beta p)^{-1}$, choose an arbitrary real number $\delta$ in $(0,\overline{\delta}(p))$ (this number will be specified at the end of the proof). For $M:=\llcorner T/\delta\lrcorner$, we define the partition $[0,T]=\cup_{m=0}^M[t_m,t_{m+1})$ with
\[
t_0=0,\,t_{M+1}=T,\,t_{m+1}-t_{m}=\delta\,\mbox{for}\,0\leq m< M.
\]
Next, for each $m$, define the family of $N$-processes $(Y^{1,N,m,\infty}_t)_{0\leq t\leq T},\dots,(Y^{N,N,m,\infty}_t)_{0\leq t\leq T}$ as: for each $1\leq i\leq N$,

\noindent
$\bullet$ Whenever $0\leq t\leq m\delta$, the path $Y^{i,N,m,\infty}_t$ is given as a weak solution to
\begin{align*}
Y^{i,N,m,\infty}_t&=Y^{i,N,m,\infty}_0+\int_{0}^{t}c(s,(Y^{i,N,m,\infty}_r)_{r\leq s})\,ds\\
&\quad+\int_0^t A(s,(Y^{i,N,m,\infty}_r)_{0\leq r\leq s})\big(B(s,(Y^{i,N,m,\infty}_r)_{0\leq r\leq s};\Ll((Y^{i,N,m,\infty}_r)_{0\leq r\leq s}))\,ds+\,dW^i_s\big);
\end{align*}
$\bullet$ Whenever $m\delta <t\leq T$,
\begin{align*}
Y^{i,N,m,\infty}_t&=Y^{i,N,m,\infty}_{m\delta}+\int_{m\delta}^{t}c(s,(Y^{i,N,m,\infty}_r)_{r\leq s})\,ds\\
&\quad +\int_{m\delta}^t A(s,(Y^{i,N,m,\infty}_r)_{0\leq r\leq s})\big(B(s,(Y^{i,N,m,\infty}_r)_{0\leq r\leq s};\overline{\nu}^{N,N}_s)\,ds+\,dW^i_s\big),
\end{align*}
for
\[
\overline{\nu}^{N,N}_t=\Ll((Y^{i,N,m,\infty}_r)_{0\leq r\leq m\delta})+\frac{1}{N}\sum_{j=1}^N\delta_{\{(Y^{j,N,m,\infty}_r)_{m\delta<  r\leq t} \}}.
\]
By construction, the sequence $\{(Y^{i,N,1,\infty}_t)_{0\leq t\leq T};\,i=1,\dots,N\}$, ..., $\{(Y^{i,N,1,\infty}_t)_{0\leq t\leq T};\,i=1,\dots,N\}$ corresponds to a partially interacting particle corresponding, for any fixed $m$, to  the McKean SDEs system \eqref{eq:McKeanVlasov_intro} up to the time $m\delta$, and integrate a mean-field interaction from $t=m\delta$ to $t=T$. Owing the uniqueness properties following \hypii and \hypiii, for $m=0$, $(Y^{1,N,0,\infty}_t,\dots,Y^{N,N,0,\infty}_t)_{0\leq t\leq T}$ corresponds to the McKean-Vlasov system \eqref{eq:ParticlePastDepend} and, for $m=M+1$, $(Y^{1,N,M+1,\infty}_t,\dots,Y^{N,N,M+1,\infty}_t)_{0\leq t\leq T}$ to the interacting particle system \eqref{eq:McKeanVlasovParticle}.
Denoting by $P^{k,m,N}$ the probability measure generated by $(Y^{1,N,m,\infty}_t)_{0\leq t\leq T},\dots,(Y^{k,N,m,\infty}_t)_{0\leq t\leq T}$ on $(\Cc([0,T];\er^d),\Bb(\Cc([0,T];\er^d)))$, by the triangular inequality,
\begin{equation}
\label{proofstp:d}
\Vert P^{k,\infty}-P^{k,N}\Vert_{TV,(0,T)}=\Vert P^{1,M+1,N}-P^{1,0,N}\Vert_{TV,(0,T)}\leq
\sum_{m=0}^{M}\Vert P^{k,m+1,N}-P^{k,m,N}\Vert_{TV,(0,T)}.
\end{equation}
%



By definition, the cost in term of an exponential martingale reduces is given by the following:   for some $0\leq m\leq M+1$, $1\leq i\leq N<\infty$, and
and, for $0\leq m\leq M-1$, using Corollary \ref{coro:DensityTwoDiff},
\begin{align*}
\frac{dP^{N,m,N}}{dP^{N,m+1,N}}
&=\exp\left\{-\sum_{i=1}^N\int_{m\delta}^{(m+1)\delta}  \triangle B^{i,N}_t\cdot \,dW^{i}_t-\frac{1}{2}\int_{m\delta}^{(m+1)\delta}  \sum_{i=1}^N
\left| \triangle B^{i,N}_t
\right|^2\,dt\right\}=Z^N_{(m+1)\delta}/Z^N_{m\delta},
\end{align*}
and
\begin{equation}\label{proofstp:e}
\begin{aligned}
\frac{dP^{N,M,N}}{dP^{N,M+1,N}}
&=\exp\left\{-\sum_{i=1}^N\int_{M\delta}^{T} \triangle B^{i,N}_t\cdot \,dW^{i}_t-\frac{1}{2}\int_{M\delta}^{T}  \sum_{i=1}^N
\left| \triangle B^{i,N}_t
\right|^2\,dt\right\}=Z^{N}_T/Z^{N}_{M\delta}.
\end{aligned}
\end{equation}
Replicating the preceding calculations from  \eqref{eq:BoundTVa} to \eqref{eq:BoundTVb}, we immediately get, for any $0\leq m\leq M-1$, and $p^*=p/(p-1)$ the conjugate of $p$,
\begin{equation}\label{proofstp:f}
\begin{aligned}
&\Vert P^{1,m+1,N}-P^{1,m,N}\Vert_{TV,(0,T)}\\
&\leq\sqrt{k} p^*\left(\EE_{\PP}\left[\left(Z^N_{(m+1)\delta}/Z^N_{m\delta}\right)^p\right]\right)^{1/p}\left(\EE_{\PP}\left[
\left(\int_{m\delta}^{(m+1)\delta} \left|\triangle B^{i,N}_t\right|^2\,dt\right)^{p^*}\right]\right)^{1/p^*}.
\end{aligned}
\end{equation}
Using Jensen's inequality and $(\mathbf{C})$, for $\lfloor p^*\rfloor$ the (least) integer part of $p^*/2$,
\begin{align*}
\EE_{\PP}\left[\left(\int_{m\delta}^{(m+1)\delta} \left|\triangle B^{i,N}_t\right|^2\,dt\right)^{p^*/2}\right]
&=\EE_{\PP}\left[\left(\left(\int_{m\delta}^{(m+1)\delta} \left|\triangle B^{i,N}_t\right|^2\,dt\right)^{\lfloor p^*/2\rfloor +1}\right)^{p^*/(2(\lfloor p^*/2\rfloor+1))}\right]\\
&\leq \left(\EE_{\PP}\left[\left(\int_{m\delta}^{(m+1)\delta} \left|\triangle B^{i,N}_t\right|^2\,dt\right)^{\lfloor p^*/2\rfloor +1}\right]\right)^{p^*/(2(\lfloor p^*/2\rfloor+1))}\\
&\leq \frac{((\lfloor p^*/2\rfloor+1)!)^{p^*/(2(\lfloor p^*\rfloor+1))}(\delta\beta)^{p^*/2}}{N^{p^*/2}}.
\end{align*}
Finally, coming back to \eqref{proofstp:f}, Proposition \ref{prop:ControlExpMart} gives:
\begin{align*}
&\Vert P^{k,m+1,N}-P^{k,m,N}\Vert_{TV,(0,T)}\leq\sqrt{k} \left(1+\exp{p^2}+\frac{2}{1-8 p\delta \beta}\right)\left(\frac{\overline{C}(p)\sqrt{\delta\beta}}{\sqrt{N}}\right),\,m=0,\dots,M-1,\\
&\overline{C}(p):=\frac{p}{p-1}((\lfloor p/(2(p-1))\rfloor+1)!)^{1/(\lfloor p/(2(p-1))\rfloor+1)}.
\end{align*}
In the same way, we get
\begin{equation}\label{proofstp:h}
\begin{aligned}
\Vert P^{k,M+1,N}-P^{k,M,N}\Vert_{TV,(0,T)}&\leq \sqrt{k} \left(1+\exp{p^2}+\frac{2}{1-8 p(T-M\delta) \beta}\right)\left(\frac{\overline{C}(p)\sqrt{(T-\delta M)\beta}}{\sqrt{N}}\right).
\end{aligned}
\end{equation}
Coming back to $\Vert P^{k,N}-P^{k,\infty}\Vert_{TV,(0,T)}$, we get
\begin{align*}
&\Vert P^{k,N}-P^{k,\infty}\Vert_{TV,(0,T)}\\
&\leq \frac{\sqrt{k}}{\sqrt{N}}\overline{C}(p)\times\left(\left(1+\exp{p^2}+\frac{2}{1-8 p\delta \beta}\right)\sqrt{\delta\beta}M +\left(1+\exp{p^2}+\frac{2}{1-8 p(T-M\delta) \beta}\right)\sqrt{(T-M\delta)\beta}\right)\\
&\leq \frac{\sqrt{k}}{\sqrt{N}}\overline{C}(p)\left(1+\exp{p^2}+\frac{2}{1-8 p\delta \beta}\right)\times\left(\sqrt{\frac{\beta}{\delta}}T+\sqrt{\delta\beta}\right).
\end{align*}
Then, choosing for instance $\delta=1/((8+\epsilon)p\beta)$ for some $\epsilon>0$, we conclude
\begin{align*}
&\Vert P^{k,N}-P^{k,N}\Vert_{TV,(0,T)}\leq C\frac{\sqrt{k}}{\sqrt{N}}(1+T\beta),\\
&C:=\inf_{p>1,\epsilon>0}\left\{ \frac{p}{p-1}\left(1+\exp{p^2}+\frac{8+\epsilon}{\epsilon}\right)
\left(\frac{((\lfloor p/(p-1)\rfloor+1)!)^{p/(p-1)\times (\lfloor p/(p-1)\rfloor+1)^{-1}}}{\sqrt{(8+\epsilon)}}\sqrt{8+\epsilon}\right)\right\}.
\end{align*}

\section{Some applications and a sufficient condition for Theorem \ref{mainthm1:LinearCaseTV}}\label{sec:SufficientConditions}
\subsection{Applications to McKean-Vlasov dynamics with bounded interaction kernel}
As an immediate consequence of Theorem \ref{mainthm1:LinearCaseTV}, we have the following propagation of chaos result for McKean's toy model:
\begin{equation*}
dX_t=\int b(t,X_t,y)\mu(t,dy)\,dt+\sigma(t,X_t) dW_t,\,\mu(t,dy)=\Ll(X_t)
\end{equation*}
\begin{corollary}\label{coro:BoundedCase} Given $b:(0,\infty)\times\er^d\times\er^d\rightarrow \er^d$ a Borel bounded function, $\sigma=\sigma(t,x)$ is a uniformly bounded and continuous, positive definite matrix-valued function in the sense that there exist $0<\lambda<\Lambda<\infty$ such that
\[
\lambda|\xi|^2\leq \xi\cdot \sigma\sigma^*(t,x)\xi\leq\Lambda|\xi|^2,\,\forall\,t\geq 0,x\in\er^d,\xi\in\er^d,
\]
let $(X^{1,N}_t,X^{2,N}_t,\dots,X^{N,N}_t)_{t\geq 0}$ and $(X^{1,\infty}_t,X^{2,\infty}_t,\dots,X^{N,\infty}_t)_{t\geq 0}$ satisfy
\begin{align}
&dX^{i,N}_t=\frac{1}{N}\sum_{j=1}^Nb(t,X^{i,N}_t,X^{j,N}_t)\,dt+\sigma(t,X^{i,N}_t) d\widetilde{W}^i_t,\label{eq:ProtoMcParticleSys}\\
&dX^{i,\infty}_t=\int b(t,X^{i,\infty}_t,y)\mu(t,dy)\,dt+\sigma(t,X^{i,\infty}_t) dW^i_t,\,\mu(t,dy)=\Ll(X^{i,\infty}_t),\label{eq:ProtoMcKeanVlasov}\\
\end{align}
where $(X^1_0,\,(W^{1}_t)_{t\geq 0}),\,\dots,(X^N_0,\,(W^{N}_t)_{t\geq 0})$ and $(\widetilde{X}^{1,N}_0,\,(\widetilde{W}^{1}_t)_{t\geq 0}),\,\dots,(\widetilde{X}^{N,N}_0,\,(\widetilde{W}^{N}_t)_{t\geq 0})$ independent copies of $(X_0,(W_t)_{t\geq 0}),\,X_0\sim\mu_0$.

Then, for any arbitrary $0<T<\infty$, we have
\[
\Vert \Ll\big((X^{1,N}_t,X^{2,N}_t,\dots,X^{N,N}_t)_{0\leq t\leq T}\big)- \Ll\big((X^{1,\infty}_t,X^{2,\infty}_t,\dots,X^{N,\infty}_t)_{0\leq t\leq T}\big)\Vert_{TV,(0,T)}\leq C(1+2\Vert \sigma^{-1} b\Vert_{L^{\infty}}T)\sqrt{\frac{k}{N}},
\]
where $C$ is given as in Theorem \ref{mainthm1:LinearCaseTV} and $\Vert \sigma^{-1} b\Vert_{L^{\infty}}:=\text{supess}_{0\leq t\leq T,\,x,y\in\er^d}\big(\sum_{l=1}^{d}|(\sigma^{-1}b)^{(l)}(t,x,y)|^2\big)^{1/2}$.
\end{corollary}
(Owing to the boundedness of the interaction kernel $b$, the wellposedness of the SDEs \eqref{eq:ProtoMcParticleSys} is immediately granted by a Girsanov transformation. For \eqref{eq:McKeanVlasov_intro}, the weak uniqueness property is immediately granted by [Jourdain \cite{Jourdain-97}, Theorem 3.2].)

As a preliminary step for the proof, let us remind the following moment inequality for the sum of i.i.d. real random variables which is a simple consequence of the moment estimates for Sub-Gaussian r.v.s' (see e.g. Bougeron, Lugosi and Massart \cite{BoLuMa-16}, Theorem 2.1) and of Hoeffding's inequality (see e.g. \cite{BoLuMa-16}, Theorem 2.8):

\begin{proposition}\label{prop:SubGaussianMoment} Let $X_1,X_2,\cdots,X_n$ be a sequence of i.i.d. random variables such that a.s. $|X_1|\leq \overline{m}<\infty$. Then, for all integer $q\geq 1$,
\[
\EE[\left(\sum_{i=1}^n\left(X_i-\EE[X_i]\right)\right)^{2q}]\leq q!(2n\overline{m}^2)^q.
\]
\end{proposition}

\begin{proof}[Proof of Corollary \ref{coro:BoundedCase}] The uniform ellipticity of $\sigma$ allowing to rewrite \eqref{eq:ProtoMcParticleSys} and \eqref{eq:ProtoMcKeanVlasov} can be rewritten into
\begin{align*}
&d\tilde{X}^{i,N}_t=\sigma(t,\tilde{X}^{i,N}_t) \big(\frac{1}{N}\sum_{j=1}^Nb(t,\tilde{X}^{i,N}_t,\tilde{X}^{j,N}_t)\,dt+d\tilde{W}^i_t\big),\\
&dX^{i,\infty}_t=\sigma(t,X^{i,N}_t) \big(\int \sigma^{-1}(t,X^{i,\infty}_t)b(t,X^{i,\infty}_t,y)\mu(t,dy)\,dt+ dW^i_t\big),\,\mu(t,dy)=\Ll(X^{i,\infty}_t).
\end{align*}

Owing to the boundedness of
$(t,x,y)\mapsto (\sigma^{-1}b)(t,x,y)$, applying Proposition \ref{prop:SubGaussianMoment} yields, for all $1\leq l\leq d$,
\begin{align*}
\EE_\PP\left[\left|\sum_{j=2}^N\left(
\sigma^{-1}(t,X^{1,\infty}_t)\left(b(t,X^{1,\infty}_t,X^{j,\infty}_t)-\int b(t,X^{1,\infty}_t,y)\,\mu(t,dy)\right)\right)
\right|^{2p}\right]\leq p!\left(2(N-1)\Vert \sigma^{-1}b\Vert^2_{L^{\infty}}\right)^p.
\end{align*}
Setting
\[
\triangle (\sigma^{-1}b)^{i,j,N}_t:=
\sigma^{-1}(t,X^{i,\infty}_t)\left(b(t,X^{i,\infty}_t,X^{j,\infty}_t)-\int b(t,X^{i,\infty}_t,y)\,\mu(t,dy)\right).
\]
Jensen's inequality yields
\begin{align*}
&\EE_\PP\left[\left(\int_{T_0}^{T_0+\delta}\left| \frac{1}{N}\sum_{j=1}^N\triangle (\sigma^{-1}b)^{i,j,N}_t\right|^{2}\,dt\right)^p\right]
%
\leq \frac{\delta^{p-1}}{N^{2p}}\int_{T_0}^{T_0+\delta}  \EE_\PP\left[\left|\sum_{j=1}\triangle (\sigma^{-1}b)^{i,j,N}_t\right|^{2p}\right]\,dt\\
&\leq \frac{\delta^{p-1}}{N^{2p}}\int_{T_0}^{T_0+\delta}  \EE_\PP\left[\left|\sum_{j=1,j\neq i}\triangle (\sigma^{-1}b)^{i,j,N}_t\right|^{2p}\right]\,dt
 +\frac{\delta^{p-1}}{N^{2p}}\int_{T_0}^{T_0+\delta}  \EE_\PP\left[\left|\triangle (\sigma^{-1}b)^{i,i,N}_t\right|^{2p}\right]\,dt\\
&\leq \frac{\delta^{p}p!(N-1)^p}{N^{p}}\Vert \sigma^{-1}b\Vert^{2p}_{L^\infty} +\frac{\delta^{p}}{N^{2p}}\Vert (\sigma^{-1}b)\Vert^{2p}_{L^\infty}.
\end{align*}
The condition $(\mathbf{C})$ is then satisfy for $\beta=2\Vert \sigma^{-1}b\Vert^2_{L^{\infty}}$ and the estimate on the total variation distance then follows from Theorem \ref{mainthm1:LinearCaseTV}.
\end{proof}

The demonstration of Corollary \ref{coro:BoundedCase}  can be easily extended to the case of Langevin dynamic yielding to the following propagation of chaos result:
\begin{corollary}\label{coro:KineticBoundedCase} Given $b:(0,\infty)\times\er^d\times\er^d\rightarrow \er^d$ a Borel bounded function and $\sigma:(0,\infty)\times\er^d\rightarrow \er^{d\times d}$, a uniformly bounded positive definite matrix-valued function,
let $((Y^{1,N}_t,V^{1,N}_t),\dots,(Y^{N,N}_t,V^{N,N}_t)_{t\geq 0}$ and $((Y^{1,\infty}_t,V^{1,\infty}_t),\dots,(Y^{N,\infty}_t,V^{N,\infty}_t)_{t\geq 0}$ satisfy
\begin{equation*}
\left\{
\begin{aligned}
&dY^{i,N}_t=V^{i,N}_t\,dt,\,\,(Y^{i,N}_0,V^{i,N})=(\widetilde{Y}^i_0,\widetilde{V}^i_0),\label{eq:ProtoLangevinParticle}\\
&dV^{i,N}_t=\frac{1}{N}\sum_{j=1}^N b(t,(Y^{i,N}_t,V^{i,N}_t),(Y^{j,N}_t,V^{j,N}_t))\,dt+\sigma(t,Y^{i,N}_t,V^{i,N}_t)d\widetilde{W}^i_t,
\end{aligned}
\right.
\end{equation*}
\begin{equation*}
\left\{
\begin{aligned}
&dY^{i,\infty}_t=V^{i,\infty}_t\,dt,\,\,(Y^{i,\infty}_0,V^{i,\infty})=(Y^i_0,V^i_0),\label{eq:ProtoLangevinMcKean}\\
&dV^{i,\infty}_t=\Big(\int b(t,(Y^{i,\infty}_t,V^{i,\infty}_t),(y,v))\,\mu(t,dy,dv)\Big)\,dt+\sigma(t,Y^{i,\infty}_t,V^{i,\infty}_t)dW^i_t,\,\mu(t)=\Ll(Y^{i,\infty}_t,V^{i,\infty}_t).
\end{aligned}
\right.
\end{equation*}
where $((Y^1_0,V^1_0),\,(W^{1}_t)_{t\geq 0}),\dots,((Y^N_0,V^N_0),\,(W^{N}_t)_{t\geq 0})$ and $((\tilde{Y}^1_0,\tilde{V}^1_0),\,(\tilde{W}^{1}_t)_{t\geq 0}),\dots,((\tilde{Y}^N_0,\tilde{V}^N_0)),\,(\tilde{W}^{N}_t)_{t\geq 0})$ are two collections of independent copies of $(Y_0,V_0)\sim\mu_0$ and $(W_t)_{t\geq 0}$.  Then, for any arbitrary $0<T<\infty$, we have
\[
\Vert  \Ll\big((Y^{1,N}_t,V^{1,N}_t),\dots,(Y^{k,N}_t,V^{k,N}_t))_{0\leq t\leq T}\big)- \Ll\big((Y^{1,\infty}_t,V^{1,\infty}_t),\dots,(Y^{k,\infty}_t,V^{k,\infty}_t)_{0\leq t\leq T}\big)\Vert_{TV,(0,T)}\leq C(1+2\beta T)\sqrt{\frac{k}{N}},
\]
where $\beta=2\text{supess}_{0\leq t\leq T,\,x,y\in\er^d}\big(\sum_{l=1}^{d}|(\sigma^{-1}b)^{(l)}(t,x,y)|^2\big)^{1/2}$.
\end{corollary}
(We refer to \cite{JabMen-19} for a detailed discussion on the wellposedness, in the weak and strong sense, of  \eqref{eq:ProtoLangevinMcKean}.)


\subsection{A sufficient condition for Theorem \ref{mainthm1:LinearCaseTV}}

In this section, we present a sufficient condition for the application of Theorem \ref{mainthm1:LinearCaseTV} which cover the corollaries \ref{coro:BoundedCase} and \ref{coro:KineticBoundedCase} as particular cases. As a warm-up, let us consider the following lemma:

\begin{lemma}\label{lem:RatePathDependent} Assume that \hypi and \hypii hold. Assume also that, for all $0\leq t<\infty$, $x\in\Cc([0,\infty);\er^d)$ $\nu\in\Pp(\Cc([0,\infty);\er^d))\mapsto B(t,x;P)$ is Lipschitz continuous w.r.t. the total variation distance; that is there exists $0<K<\infty$ such that $P,Q\in\Pp(\Cc([0,\infty);\er^d))$, $0\leq t<\infty$, $x\in\Cc([0,\infty);\er^d)$,
\begin{equation}\label{cond:TVLip}
\left|B(t,x;P)-B(t,x;Q)\right|\leq K\Vert P-Q\Vert_{TV,(0,t)}.
\end{equation}
Assume finally that the following centering (conditional) property holds:
\[
\EE_{\PP}\left[B\big(t,X^{i,\infty};\frac{1}{N-1}\sum_{j=1,j\neq i}^{N}\delta_{\{X^{i,\infty}\}}\big)\,\Big{|}\,X^{i,\infty}\right]=B(t,X^{i,\infty};\Ll(X^{i,\infty}))
\]
Then the condition $(\mathbf{C})$ is satisfied for $\beta= 4K$.
\end{lemma}
Prior to the proof let us recall the notion of functions with bounded difference and an annex concentration property:
\begin{definition}
  Let $E$ be some measurable space. A function $f:E^n\rightarrow \er$ is said to have the bounded difference property if, there exists $c_1,c_2,\cdots,c_n>0$ such that for all $(x_1,x_2,\cdots,x_n),(y_1,y_2,\cdots,y_n)\in E^n$, we have for all $1\leq i\leq n$,
\[
\left|f(x_1,\cdots,x_{i-1},x_i,x_{i+1},\cdots,x_n)-f(x_1,\cdots,x_{i-1},y_i,x_{i+1},\cdots,x_n)\right|\leq c_i.
\]
\end{definition}

\begin{theorem}[Bounded Difference Inequality, \cite{BoLuMa-16},  Theorem $6.2$]\label{thm:BoundedDifferenceIneq} Let $E$ be some measurable space, $(Y_1,\cdots,Y_n)$ be a family of $E$-valued i.i.d. random variables and let $f:E^n\rightarrow \er$ be some function satisfying the bounded difference property. Then
\[
\mathbf{Y}=f(Y_1,\cdots,Y_n)
\]
satisfies: for all $t\geq 0$,
\[
\max\left(\PP\left(\mathbf{Y}-\EE[\mathbf{Y}]\geq t\right),\PP\left(\mathbf{Y}-\EE[\mathbf{Y}]\leq -t\right)\right)\leq \exp\{-\frac{t^2}{2\nu}\},
\]
for $\nu=\sum_{i=1}^n (c_i)^2/4$.
\end{theorem}
In particular, the above ensure the following moment estimates: For all integer $k\geq 1$,
\begin{equation}\label{proofstp:j}
\EE\left[\left(\mathbf{Y}-\EE[\mathbf{Y}]\right)^{2k}\right]\leq k!(4\nu)^k.
\end{equation}

\begin{proof}[Proof of Lemma \ref{lem:RatePathDependent}] Fix $t\geq 0$ and $\nu$ an arbitrary probability measure on $\Cc([0,\infty);\er^d)$ and define
the family of mappings
\[
f^{(l)}_i:\mathbf{x}^N=(x_1,x_2,\cdots,x_n)\in \Cc([0,\infty);\er^d)\mapsto f_i(\mathbf{x}^N)= \left(B^{(l)}(t,x_i,\mu^{-i,N}(\mathbf{x}^N))-B^{(l)}(t,\mathbf{x},\nu)\right)\in\er,\,1\leq l\leq m,
\]
for $1\leq i\leq N$, $\mu^{-i,N}(\mathbf{x}^N)=\frac{1}{N}\sum_{j=1,j\neq i}^N\delta_{\{x_j\}}$ the empirical measure related to $\mathbf{x}^N$ deprived of $x_i$. For any $i$, $l$, observe that the Lipschitz condition \eqref{cond:TVLip} implies that:
\begin{align*}
&\left|f^{(l)}_i(x_1,\cdots,x_{k-1},x,x_{k+1},\cdots,x_n)-f^{(l)}_i(x_1,\cdots,x_{k-1},y,x_{k+1},\cdots,x_n)\right|\\
&=\left|B^{(l)}\big(t,x_i,\frac{1}{N}\sum_{j=1,j\neq i,k}^N\delta_{\{x_j\}}+ \frac{1}{N}\delta_{\{x\}}\big)-
B^{(l)}\big(t,x_i,\frac{1}{N}\sum_{j=1,j\neq i,k}^N\delta_{\{x_j\}}+ \frac{1}{N}\delta_{\{y\}}\big)
\right|\\
&\leq K\Vert \frac{1}{N}\delta_{\{x\}}-\frac{1}{N}\delta_{\{y\}}\Vert_{TV,(0,T)}\leq\frac{K}{N}.
\end{align*}
so that each of the $f_k$'s satisfies a bounded difference property with coefficients $c_i:=K/N$ for all $1\leq i\leq N$.
Applying \eqref{proofstp:j} with $\nu=\sum_{i=1}^N(c_i)^2/4=K^2/4N$, it follows that
\begin{align*}
\EE_\PP\left[\left|\left(B^{(l)}\Big(t,X^{i,\infty},\frac{1}{N}\sum_{j=1,j\neq i}^N\delta_{\{(X^{k,N}_t)_{0\leq t\leq T} \}}\Big)-B^{(l)}\Big(t,X^{i,\infty},\Ll(X^{i,\infty})\Big)\right)\right|^{2p}\right]\leq p! \frac{K^{2p}}{N^p},
\end{align*}
from which we deduce that
\begin{align*}
&\EE_\PP\left[\left|B^{(l)}\Big(t,X^{i,\infty},\frac{1}{N}\sum_{j=1}^N\delta_{\{(X^{k,N}_t)_{0\leq t\leq T} \}}\Big)-B^{(l)}\Big(t,X^{i,\infty},\Ll(X^{i,\infty})\Big)\right|^{2p}\right]\\
&\leq 2^{2p-1}\EE_\PP\left[\left|B^{(l)}\Big(t,X^{i,\infty},\frac{1}{N}\sum_{j=1,j\neq i}^N\delta_{\{X^{k,N}_.\}}\Big) -B^{(l)}\Big(t,X^{i,\infty},\Ll(X^{i,\infty})\Big)\right|^{2p}\right]\\
&\quad +2^{2p-1}
\EE_\PP\left[\left|B^{(l)}\Big(t,X^{i,\infty},\frac{1}{N}\sum_{j=1}^N\delta_{\{X^{k,N}_.\}}\Big)-B^{(l)}\Big(t,X^{i,\infty},\frac{1}{N}\sum_{j=1,j\neq i}^N\delta_{\{X^{k,N}_.\}}\Big)\right|^{2p}\right]\\
&\leq p! \frac{2^{2p-1}K^{2p}}{N^p}+\frac{2^{2p-1}K^p}{N^{2p}}\leq  p! \frac{4^{p}K^p}{N^p}.
\end{align*}
Therefore,
\begin{equation*}
 \EE_\PP\left[\left(\int_{T_0}^{T_0+\delta}\left|B(t,X^{i,\infty},\mu^{N,\infty})-B(t,X^{i,\infty},\Ll(X^{i,\infty}))\right|^{2}\,dt\right)^p\right]\leq \frac{(4\delta m K^2)^p}{N^p}.
\end{equation*}
\end{proof}

The core argument of Lemma \ref{lem:RatePathDependent} relies mostly on the centering property regularity of the drift component $(A B)$ in its measure argument formulated in terms of an analog of the linear derivative functional linear (see e.g. [\cite{Kolokolstov-10}, Appendix $F$], [\cite{CarDel-18a}, Section 5.4]) here below  set on the sample space $\Cc([0,T];\er^d)$ :

\begin{definition}\label{def:FlatDerivative} The $\er^m$-valued functional $B=\left(B^{(1)},B^{(2)},\dots,B^{(m)}\right)$ is said to admit a bounded second order flat derivative if, for all $1\leq l\leq m$ there exist two measurable bounded functionals:
\[
\frac{d B^{(l)}}{dm}=:\in [0,T]\times\Cc([0,T];\er^d)\times \Pp(\Cc([0,T];\er^d))\times \Cc([0,T];\er^d)\rightarrow \er,
\]
\[
\frac{d^2 B^{(l)}}{dm^2}:(t,x,m;\omega_1,\omega_2)\in [0,T]\times\Cc([0,T];\er^d)\times \Pp(\Cc([0,T];\er^d))\times \Cc([0,T];\er^d)\times \Cc([0,T];\er^d)\rightarrow \er,
\]
such that, for all $0<T<\infty$, $0\leq t\leq T$, $x\in\Cc([0,T];\er^d)$, $P,Q\in\Pp(\Cc([0,T];\er^d)$,
\[
B^{(l)}(t,x,Q)-B^{(l)}(t,x,P) =\int_{0}^{1}\int_{\omega\in\Cc([0,T];\er^d)}\frac{d B^{(l)}}{dm}(t,x,(1-\alpha)P+\alpha Q;\omega)\left(Q(d\omega)-P(d\omega)\right)\,d\alpha,
\]
and, for all $0<T<\infty$, $0\leq t\leq T$, $x\in\Cc([0,T];\er^d)$, $P,Q\in\Pp(\Cc([0,T];\er^d)$, $\omega\in\Cc([0,T];\er^d)$,
\begin{align*}
&\frac{d B^{(l)}}{dm}(t,x,Q;\omega)-\frac{d B}{dm}(t,x,P;\omega)\\
&=\int_{0}^{1} \int_{\tilde{\omega}\in\Cc([0,T];\er^d)}\frac{d^2 B^{(l)}}{d m^2}(t,x,(1-\alpha)P+\alpha Q;\omega,\tilde{\omega})\left(Q(d\tilde{\omega})-P(d\tilde{\omega})\right)\,d\alpha.
\end{align*}
where $(1-\alpha)P+\alpha Q,\,0\leq \alpha\leq 1$ is the set of probability measures given by the convex interpolations between $P$ and $Q$.
\end{definition}

\begin{proposition}\label{prop:DifferentiabilityCondition} Assume that \hypi and \hypii hold and that for all $0\leq t\leq T$, $x\in\Cc([0,T];\er^d)$, $\mu\in\Pp(\Cc[0,T];\er^d)\mapsto B(t,x,\mu)$ admits a uniformly bounded second order derivative in the sense of Definition \ref{def:FlatDerivative}. Then the condition $\mathbf{(C)}$ in Theorem \ref{mainthm1:LinearCaseTV} holds.
\end{proposition}
\begin{proof} For any $1\leq l\leq m$, using $\frac{d B^{(l)}}{d m}$, we have
\begin{align*}
&\triangle B^{i,N,(l)}_t:=B^{(l)}(t,(X^{i,\infty}_r)_{0\leq r\leq t},\overline{\nu}^N_t)-B^{(l)}(t,(X^{i,\infty}_r)_{0\leq r\leq t},\Ll((X^{i,\infty}_r)_{0\leq r\leq t}))\\
%
\end{align*}
for $\overline{\nu}^{\alpha,N}_t=(1-\alpha)\overline{\nu}^{N}_t+\alpha\Ll((X^{i,\infty}_r)_{0\leq r\leq t})$.
In the first sum, for fixed $j$, define the (partial) empirical measure $\overline{\nu}^{-j,N}_t=\frac{1}{N-1}\sum_{l=1,l\neq j}^{N}\delta_{\{(X^{l,\infty}_r)_{0\leq r\leq t}\}}$. Adding and subtracting to the above,
\begin{align*}
&\frac{1}{N}\sum_{j=1}^N\int_{0}^{1} \frac{dB^{(l)}}{dm}(t,(X^{i,\infty}_r)_{0\leq r\leq t},\overline{\nu}^{-j,\alpha,N}_t
;(X^{j,\infty}_r)_{0\leq r\leq t})\\
&\quad-\int_{0}^{1}\int_{\omega\in\Cc([0,T];\er^d)}\frac{dB^{(l)}}{dm}(t,(X^{i,\infty}_r)_{0\leq r\leq t},\overline{\nu}^{-j,\alpha,N}_t
)(\omega)\Ll((X^{i,\infty}_r)_{0\leq r\leq t})(d\omega)\,d\alpha,
\end{align*}
for
\[
\overline{\nu}^{-j,\alpha,N}_t=(1-\alpha)\overline{\nu}^{-j,N}_t+\alpha\Ll((X^{i,\infty}_r)_{0\leq r\leq t}),
\]
we have the decomposition:
\[
\triangle B^{i,N,(l)}_t=I^{i,N,(l)}_t+J^{i,N,(l)}_t+K^{i,N,(l)}_t,
\]
where
\begin{align*}
I^{i,N,(l)}_t&:=\frac{1}{N}\sum_{j=1}^N\int_{0}^{1}\left( \frac{dB^{(l)}}{dm}(t,(X^{i,\infty}_r)_{0\leq r\leq t},\overline{\nu}^{\alpha,N}_t
;(X^{j,\infty}_r)_{0\leq r\leq t})-\frac{dB^{(l)}}{dm}(t,(X^{i,\infty}_r)_{0\leq r\leq t},\overline{\nu}^{-j,\alpha,N}_t
;(X^{j,\infty}_r)_{0\leq r\leq t})\right)\,d\alpha,
\end{align*}
\begin{align*}
J^{i,N,(l)}_t&:=\frac{1}{N}\sum_{j=1}^N\int_{0}^{1} \frac{dB^{(l)}}{dm}(t,(X^{i,\infty}_r)_{0\leq r\leq t},\overline{\nu}^{-j,\alpha,N}_t
;(X^{j,\infty}_r)_{0\leq r\leq t})\\
&\quad-\int_{0}^{1}\int_{\omega\in\Cc([0,T];\er^d)}\frac{dB^{(l)}}{dm}(t,(X^{i,\infty}_r)_{0\leq r\leq t},\overline{\nu}^{-j,\alpha,N}_t
)(\omega)\Ll((X^{i,\infty}_r)_{0\leq r\leq t})(d\omega)\,d\alpha,
\end{align*}
\begin{align*}
&K^{i,N,(l)}_t\\
&:=\int_{0}^{1}\int_{\omega\in\Cc([0,T];\er^d)}\left(\frac{dB^{(l)}}{dm}(t,(X^{i,\infty}_r)_{0\leq r\leq t},\overline{\nu}^{-j,\alpha,N}_t
;\omega)-\frac{dB^{(l)}}{dm}(t,(X^{i,\infty}_r)_{0\leq r\leq t},\overline{\nu}^{\alpha,N}_t
;\omega)\right)\Ll((X^{i,\infty}_r)_{0\leq r\leq t})(d\omega)\,d\alpha,
\end{align*}
Using the second order derivative $d^2B/dm^2$ and since
\[
\overline{\nu}^{\alpha,N}_t(d\omega)-\overline{\nu}^{-j,\alpha,N}_t(d\omega)=\frac{1}{N}\delta_{\{(X^{j,N}_r)_{0\leq r\leq t}\}}+\frac{1}{N(N-1)} \sum_{l=1,l\neq j}^N\delta_{\{(X^{l,N}_r)_{0\leq r\leq t}\in d\omega\}}
\]
we immediately get for $I^{i,N}_t$:
\begin{align*}
I^{i,N,(l)}_t&=\frac{1}{N}\sum_{j=1}^N\int_{0}^{1}\int_{0}^1 \frac{d^2B^{(l)}}{dm^2}(t,(X^{i,\infty}_r)_{0\leq r\leq t},(1-r)\overline{\nu}^{\alpha,N}_t+r\overline{\nu}^{\alpha,N}_t
;(X^{j,\infty}_r)_{0\leq r\leq t};\tilde{\omega}) \left(\overline{\nu}^{\alpha,N}_t(d\tilde{\omega})-\overline{\nu}^{-j,\alpha,N}_t(d\tilde{\omega})\right)\,d\alpha\,dr\\
&=\frac{1}{N^2}\sum_{j=1}^N\int_{0}^{1}\int_{0}^1 \frac{d^2B^{(l)}}{dm^2}(t,(X^{i,\infty}_r)_{0\leq r\leq t},(1-r)\overline{\nu}^{\alpha,N}_t+r\overline{\nu}^{\alpha,N}_t
;(X^{j,\infty}_r)_{0\leq r\leq t},(X^{j,\infty}_r)_{0\leq r\leq t})\,d\alpha\,dr\\
&\quad + \frac{1}{N^2(N-1)}\sum_{j=1}^N\sum_{l=1,l\neq j}^N\int_{0}^{1}\int_{0}^1 \frac{d^2B^{(l)}}{dm^2}(t,(X^{i,\infty}_r)_{0\leq r\leq t},(1-r)\overline{\nu}^{\alpha,N}_t+r\overline{\nu}^{\alpha,N}_t
;(X^{j,\infty}_r)_{0\leq r\leq t},(X^{l,\infty}_r)_{0\leq r\leq t})\,d\alpha\,dr.
\end{align*}
In the same way,
\begin{align*}
&K^{i,N,(l)}_t\\
&=\frac{1}{N}\sum_{j=1}^N\int_{0}^{1}\int_{0}^1 \frac{dB^{(l)}}{dm}(t,(X^{i,\infty}_r)_{0\leq r\leq t},(1-r)\overline{\nu}^{\alpha,N}_t+r\overline{\nu}^{\alpha,N}_t
;(X^{j,\infty}_r)_{0\leq r\leq t};\tilde{\omega}) \left(\overline{\nu}^{\alpha,N}_t(d\tilde{\omega})-\overline{\nu}^{-j,\alpha,N}_t(d\tilde{\omega})\right)\,d\alpha\,dr\\
&=\frac{1}{N^2}\sum_{j=1}^N\int_{0}^{1}\int_{0}^1 \frac{d^2B^{(l)}}{dm^2}(t,(X^{i,\infty}_r)_{0\leq r\leq t},(1-r)\overline{\nu}^{\alpha,N}_t+r\overline{\nu}^{\alpha,N}_t
;\omega,(X^{j,\infty}_r)_{0\leq r\leq t})\Ll((X^{j,\infty}_r)_{0\leq r\leq t})\,d\alpha\,dr\\
&\quad + \frac{1}{N^2(N-1)}\sum_{j=1}^N\sum_{l=1,l\neq j}^N\int_{0}^{1}\int_{0}^1 \frac{d^2B^{(l)}}{dm^2}(t,(X^{i,\infty}_r)_{0\leq r\leq t},(1-r)\overline{\nu}^{\alpha,N}_t+r\overline{\nu}^{\alpha,N}_t
;\omega,(X^{l,\infty}_r)_{0\leq r\leq t})\Ll((X^{j,\infty}_r)_{0\leq r\leq t})\,d\alpha\,dr.
\end{align*}
These estimates ensure directly that
\begin{align*}
\EE\left[\left(\int_{T_0}^{T_0+\delta}\left|I^{i,N,(l)}_t\right|^2\,dt\right)^p\right]\leq \frac{2^p\delta^p}{N^p}\Vert \frac{d^2 B}{dm ^2}\Vert^{2p}_{L^\infty},
\end{align*}
and
\begin{align*}
\EE\left[\left(\int_{T_0}^{T_0+\delta}\left|K^{i,N,(l)}_t\right|^2\,dt\right)^p\right]\leq \frac{2^p\delta^p}{N^p}\Vert \frac{d^2 B}{dm ^2}\Vert^{2p}_{L^\infty}.
\end{align*}
The final component $J^{i,N,(l)}$ can be estimated in the same way as in the proof of Lemma \ref{lem:RatePathDependent}.
\end{proof}

\paragraph{Acknowledgement:} 
This article was prepared within the framework of the Russian Academic Excellence Project '5-100'. The author is thankful to Lukasz Szpruch and Paul-Eric Chaudru de Raynal for having pointed out the use of linear functional derivative to derive the sufficient condition in Proposition \ref{prop:DifferentiabilityCondition}, and to Alexander Veretennikov for very fruitful discussions over the past year.

\section{Appendix}
\textbf{Carlen and Kr\'ee's optimal martingale moment control}:

\begin{theorem}[Carlen and Kr\'ee \cite{CarKre-91}, Theorem $A$]\label{thm:CarlenKree} For $p\geq 1$, define
\[
b_p=\sup_{(M_t)_{t\geq 0}}\left\{\frac{\EE\left[(M_t)^p\right]^{1/p}}{\EE\left[(\sqrt{\langle M\rangle_t})^p\right]^{1/p}}\right\},
\]
where the supremum is taken over the set of real valued bounded and continuous martingales $(M_t)_{t\geq 0}$. Then
\[
\sup_{p\geq 1}\frac{b_p}{\sqrt{p}}=2.
\]
\end{theorem}
The boundedness condition, assumed in Carlen and Kr\'ee \cite{CarKre-91}, can be easily dropped, thanks to a truncation argument, to state the generic inequality:
\begin{equation}\label{proofst:h}
\EE\left[(M_t)^p\right]^{1/p}\leq 2\sqrt{p}\EE\left[(\sqrt{\langle M\rangle_t})^p\right]^{1/p}\,\text{whenever}\,\EE\left[(\sqrt{\langle M\rangle_t})^p\right]<\infty.
\end{equation}
Indeed, given $(M_t)_{t\geq 0}$ a continuous $L^p$-finite martingale and introducing the stopping time $\tau_\lambda=\inf\{t>0\,:\,|M_t|\geq \lambda\}$, the truncated process $(M_{t\wedge\tau_\lambda};\,t\geq 0)$ is bounded, so that
\[
\EE\left[(M_{t\wedge \tau_\lambda})^p\right]^{1/p}\leq 2\sqrt{p}\EE\left[(\sqrt{\langle M\rangle_{t\wedge \tau_\lambda}})^p\right]^{1/p}.
\]
Taking the limit $\lambda\rightarrow \infty$, we conclude \eqref{proofst:h}

\noindent
\textbf{Proof of \eqref{proofstp:i}:}

From this proposition, we deduce the following corollary that can be simply deduced from [Theorem $7.7$, Lipster and Shiryaev \cite{LipShi-01}]:
\begin{corollary}\label{coro:DensityTwoDiff} Let $(\zeta^1_t)_{0\leq t\leq T}$ and $(\zeta^2_t)_{0\leq t\leq T}$ be two It\^o diffusion processes defined a filtered probability space $(\Omega,\Ff,(\Ff_t)_{t\geq 0},\PP)$, satisfying
\[
d\zeta^i_t=\alpha_i(t,\zeta^i)\,dt+dW^i_t,\,\zeta_0=0,\,0\leq t\leq T,,\,i=1,2,
\]
Then assuming that
\[
\PP\left(\int_0^T\left|\alpha_1(t,\zeta^1)\right|^2\,dt +\int_0^T\left|\alpha_2(t,\zeta^2)\right|^2\,dt<\infty\right)=1,
\]
and
\[
\PP\left(\int_0^T\left|\alpha_1(t,W^1)\right|^2\,dt+\int_0^T\left|\alpha_2(t,W^2)\right|^2\,dt<\infty\right)=1,
\]
the probability measures $P_{\zeta_1}$ and $P_{\zeta_2}$ are equivalent and
\[
\frac{dP_{\zeta^1}}{dP_{\zeta^2}}(T,\zeta^2)=\exp\left\{-\int_0^T \left(\alpha_1(t,\zeta_2)-\alpha_2(t,\zeta_2)\right)\cdot \,d\zeta^2_t-\frac{1}{2}\int_0^T \left|\alpha_1(t,\zeta^2)-\alpha_2(t,\zeta^2)\right|^2\,dt
\right\},
\]
\[
\frac{dP_{\zeta^2}}{dP_{\zeta^1}}(T,\zeta^1)=\exp\left\{-\int_0^T \left(\alpha_2(t,\zeta_1)-\alpha_2(t,\zeta_1)\right)\cdot \,d\zeta^1_t-\frac{1}{2}\int_0^T \left|\alpha_2(t,\zeta^1)-\alpha_1(t,\zeta^1)\right|^2\,dt\right\}.
\]
\end{corollary}
Applying the preceding corollary to \eqref{eq:McKeanVlasovParticle} and \eqref{eq:Nparticles},
we deduce \eqref{proofstp:i} by applying two successive Girsanov transformations, first mapping the $\er^{dN}$-valued process:
\begin{equation*}
(X^{1,\infty}_t,\dots,X^{N,\infty}_t)_{0\leq t\leq T},
\end{equation*}
into a system of $N$ (independent) copies of the solution to \eqref{eq:IntermediateSDE}. The interaction between the component is then introduced by a second Girsanov transformation yielding to \eqref{eq:Nparticles}.

\end{document}